\documentclass[12pt,reqno]{amsart}
\usepackage{amscd,amsmath,amsthm,amssymb}
\usepackage{color}
\usepackage{tikz}
\usepackage{url}
\usepackage{circuitikz}
\usepackage{float}

\newtheorem{Theorem}{Theorem}[section]
\newtheorem{Lemma}[Theorem]{Lemma}
\newtheorem{Corollary}[Theorem]{Corollary}
\newtheorem{Proposition}[Theorem]{Proposition}

\newtheorem{Definition}[Theorem]{Definition}

\newtheorem{Question}[Theorem]{Question}

\def\qed{\ifhmode\textqed\fi
	\ifmmode\ifinner\quad\qedsymbol\else\dispqed\fi\fi}
\def\textqed{\unskip\nobreak\penalty50
	\hskip2em\hbox{}\nobreak\hfill\qedsymbol
	\parfillskip=0pt \finalhyphendemerits=0}
\def\dispqed{\rlap{\qquad\qedsymbol}}

\def\Ker{\textup{Ker}}
\def\ZZ{\mathbb{Z}}

\def\m{\mathfrak{m}}
\def\n{\mathfrak{n}}

\def\Tor{\textup{Tor}}

\def\supp{\textup{supp}}
\def\reg{\textup{reg}}

\def\rank{\textup{rank}}

\def\sgn{\textup{sgn}}

\begin{document}
	
	\title{On the gradient of a monomial ideal}
	\author{Antonino Ficarra}
	
	\address{Antonino Ficarra, BCAM -- Basque Center for Applied Mathematics, Mazarredo 14, 48009 Bilbao, Basque Country -- Spain, Ikerbasque, Basque Foundation for Science, Plaza Euskadi 5, 48009 Bilbao, Basque Country -- Spain}
	\email{aficarra@bcamath.org,\,\,\,\,\,\,\,\,\,\,\,\,\,antficarra@unime.it}
	
	%\thanks{}
	
	\begin{abstract}
		Let $K$ be a field of characteristic zero, let $I \subset S = K[x_1,\dots,x_n]$ be a homogeneous ideal, and let $\partial(I)$ be its gradient ideal. We study the relationship between $\mathrm{reg}\,I$ and $\mathrm{reg}\,\partial(I)$. While earlier work by Bus\'e, Dimca, Schenck, and Sticlaru showed these regularities are generally incomparable for hypersurface ideals, we prove they remain incomparable even for monomial ideals with linear resolution, answering a question of J. Herzog. In fact, for any integers $a \in \mathbb{Z}$ and $b \ge - 1$, we construct monomial ideals $I$ and $J$ such that $\mathrm{reg}\,I - \mathrm{reg}\,\partial(I) = a$, $\mathrm{reg}\,\partial(J) - \mathrm{reg}\,J = b$ and $J$ has linear resolution. We introduce monomial ideals with \emph{differential linear resolution} as those monomial ideals whose all iterated gradient ideals have linear resolution. We prove that polymatroidal ideals, equigenerated (strongly) stable ideals, powers of edge ideals with linear resolution, complementary edge ideals with linear resolution, and certain equigenerated squarefree monomial ideals with many generators satisfy this property.
	\end{abstract}
	
	\maketitle
	\section{Introduction}
	Let $f\in\mathbb{C}[x_1,\dots,x_n]$ be a homogeneous polynomial of degree $d$. The \textit{Jacobian ideal} of $f$ is the ideal $\textup{Jac}(f)=(\partial_{x_1}f,\dots,\partial_{x_n}f)$ generated by the partial derivatives of $f$. The ring $M(f)=\mathbb{C}[x_1,\dots,x_n]/\textup{Jac}(f)$ is often called the \textit{Milnor algebra} of $f$. It encodes the singular locus of the hypersurface $V(f)\subset\mathbb{P}^{n-1}$. There is a strong interplay between the algebraic properties of $M(f)$ and the geometry of $V(f)$. For instance, it is well-known that $f$ is a free divisor if and only if $\textup{Jac}(f)$ is a perfect ideal of codimension two. When $V(f)$ is smooth, then $M(f)$ is an Artinian complete intersection and plays a central role in the Hodge algebra of $V(f)$.
	
	In \cite{BDSS}, Bus\'e, Dimca, Schenck, and Sticlaru compared the Castelnuovo-Mumford regularities $\reg(f)=d$ and $\reg\,\textup{Jac}(f)$, where $f\in\mathbb{C}[x_1,\dots,x_n]$ is a homogeneous polynomial of degree $d$. Despite the naive expectation that $\reg\,\textup{Jac}(f)\le d-1$, these authors proved that the regularity of $\textup{Jac}(f)$ is bounded above by $n(d-2)$ under certain assumptions and can even grow quadratically with $d$.
	
	More generally, for a homogeneous ideal $I\subset S=K[x_1,\dots,x_n]$, where $K$ is a field, we can define the \textit{gradient ideal} of $I$ as the ideal $\partial(I)$ generated by all the partial derivatives of any set of minimal generators of $I$. In characteristic zero, the definition of $\partial(I)$ does not depend on the particular generating set of $I$ chosen. Gradient ideals appear in the work of Herzog and Huneke on Golodness of various kind of powers of homogeneous ideals \cite{HHun}, as well as in the work of Ahangari Maleki \cite{Maleki}.
	
	We would like to compare the regularities $\reg\,I$ and $\reg\,\partial(I)$ for a homogeneous ideal $I\subset S$. Of course, due to the work of \cite{BDSS}, in general we cannot expect that $\reg\,\partial(I)\le\reg\,I$. On the other hand, if $I$ is a monomial ideal, some initial evidence lead to believe that such a nice comparison could hold. Let $\mathcal{G}(I)$ be the minimal monomial generating set of a monomial ideal $I\subset S$. In this situation, if the characteristic of $K$ is zero or is strictly larger than any exponent appearing in the minimal monomial generators of $I$, then $\partial(I)=\partial^*(I)$ where
	$$
	\partial^*(I)\ =\ (u/x_i:\ u\in\mathcal{G}(I),\,i\in\supp(u)),$$
	and $\supp(u)=\{i:x_i\ \textup{divides}\ u\}$ is the \textit{support} of $u$. Hereafter, we always assume that the equality $\partial(I)=\partial^*(I)$ holds for any monomial ideal $I$ we will consider.
	
	It is easily seen that if $I$ is principal monomial ideal generated in degree $d$, then $\partial(I)$ is an ideal of Veronese type and so $\reg\,\partial(I)=\reg\,I-1=d-1<\reg\,I$.
	
	It is obvious that $\partial^*(I)=\sum_{i=1}^n(I:x_i)$. It is well-known that $\reg(I:x_i)\le\reg\,I$ if $I\subset S$ is a monomial ideal \cite[Lemma 4.2]{Fak}.
	
	Hence, in view of these facts one may naively expect that $\reg\,\partial(I)\le\reg\,I$ for any monomial ideal $I\subset S$. In Theorem \ref{Thm:reg-a} we disprove this expectation and we show that the difference $\reg\,I-\reg\,\partial(I)$ can be any relative integer in $\ZZ$.
	
	On the other hand, it is quite difficult to find monomial ideals $I\subset S$ with linear resolution such that $\reg\,\partial(I)>\reg\,I$. During a mathematical conversation, J\"urgen Herzog posed the following question to the author of this paper:
	\begin{Question}\label{Quest:Her}
		Let $I\subset S$ be a monomial ideal with $d$-linear resolution. Is it true that $\reg\,\partial(I)\le d$?
	\end{Question}
	
	There are many classes $\mathcal{C}$ of monomial ideals with linear resolution which are closed under the gradient operation. That is, if $I\in\mathcal{C}$ then $\partial(I)\in\mathcal{C}$. For such classes, the previous question holds trivially because $\reg\,\partial(I)=\reg\,I-1$. Among these families we have the class of polymatroidal ideals (Theorem \ref{Thm:partial-polym}) and the class of equigenerated (strongly) stable monomial ideals (Proposition \ref{Prop:stable}).
	
	For a homogeneous ideal $I\subset S$, we put $\partial^0(I)=I$, $\partial^1(I)=\partial(I)$ and for any integer $\ell>1$ we put $\partial^\ell(I)=\partial(\partial^{\ell-1}(I))$ recursively.
	
	We say that a homogeneous ideal $I\subset S$ has \textit{differential linear resolution} if $\partial^\ell(I)$ has linear resolution for all $0\le\ell\le\alpha(I)$. The classes mentioned before enjoy this property, as well as all squarefree monomial ideals $I\subset S$ having $d$-linear resolution with $d\in\{0,1,2,n-2,n-1,n\}$ (Theorem \ref{Thm:sqfree-d}), equigenerated squarefree monomial ideals with many generators satisfying a certain numerical condition (Theorem \ref{Thm:Kruskal}) and the powers of edge ideals with linear resolution (Theorem \ref{Thm:edgeideal}).
	
	Despite all the positive evidence supporting it, Question \ref{Quest:Her} has in fact a negative answer. In Theorem \ref{Thm:linRes-Partial}, we prove that for an arbitrary monomial ideal $I\subset S$ with linear resolution the difference $\reg\,\partial(I)-\reg\,I$ can be any integer $\ge-1$.
	
	Finally, we would like to mention that we expect that equigenerated monomial ideals $I\subset S$ with many generators, namely such that $|\mathcal{G}(I)|\ge\binom{n}{\alpha(I)}-2\alpha(I)+1$ where $\alpha(I)$ is the generating degree of $I$, as well as the powers of any complementary edge ideal with linear resolution, have differential linear resolution.
	
	\section{Comparison of $\reg\,I$ and $\reg\,\partial(I)$}\label{sec1}
	
	In this first section, we compare the Castelnuovo-Mumford regularities of a monomial ideal and its gradient.
	
	Let $S=K[x_1,\dots,x_n]$ be the standard graded polynomial ring over a field $K$, and let $I\subset S$ be a monomial ideal. For an integer $n\ge1$, we put $[n]=\{1,2,\dots,n\}$. We always tacitly assume that $\text{char}(K)=0$ or that $\text{char}(K)$ is big enough so that
	$$
	\partial(I)\ =\ \partial^*(I)\ =\ \sum_{i\in[n]}(I:x_i).
	$$
	
	We denote by $\mathcal{G}(I)$ the minimal monomial generating set of $I$. Let $\m=(x_1,\dots,x_n)$ be the graded maximal ideal of $S$. For a homogeneous ideal $I\subset S$ we put
	$$
	\alpha(I)\ =\ \min\{d:(I/\m I)_d\ne0\},\quad\quad\,\,\,\omega(I)\ =\ \max\{d:(I/\m I)_d\ne0\}.
	$$
	
	If, in addition, $I\subset S$ is a monomial ideal, then $\alpha(I)=\min\{\deg(u):u\in\mathcal{G}(I)\}$ and $\omega(I)=\max\{\deg(u):u\in\mathcal{G}(I)\}$.
	
	Let
	$$
	\reg\,I\ =\ \max\{j-i:\ \beta_{i,j}(I)=\dim_K\Tor^S_i(K,I)_j\ne0\}
	$$
	be the \textit{Castelnuovo-Mumford regularity} of $I$. We always have $\reg\,I\ge\omega(I)\ge\alpha(I)$. If $\reg\,I=\omega(I)=\alpha(I)=d$, we say that $I$ has a \textit{$d$-linear resolution}, or simply that $I$ has linear resolution, if we do not want to specify the generating degree $d$ of $I$.
	
	Given ${\bf a}\in\ZZ_{\ge0}^n$, we put ${\bf x^a}=\prod_{i\in[n]}x_i^{a_i}$ and $|{\bf a}|=\deg({\bf x^a})=a_1+\dots+a_n$. Let $d$ be an integer with $d\le|{\bf a}|$. The ideal of \textit{Veronese type} $({\bf a},d)$ is the ideal defined as
	$$
	I_{n,{\bf a},d}\ =\ ({\bf x^b}:\ |{\bf b}|=d,\ {\bf b}\le{\bf a}).
	$$
	
	Note that $\partial(({\bf x^a}))=I_{n,{\bf a},|{\bf a}|-1}$ and this ideal has linear resolution. Using this elementary observation we obtain
	\begin{Proposition}
		Let $I\subset S$ be a monomial ideal. Then
		$$
		\alpha(I)-1\ \le\ \reg\,\partial(I)\ \le\ \sum_{u\in\mathcal{G}(I)}\deg(u)-2|\mathcal{G}(I)|+1,
		$$
		and equality holds for the upper bound if $I$ is a complete intersection.
	\end{Proposition}
	\begin{proof}
		Let $\mathcal{G}(I)=\{u_1,\dots,u_m\}$ with $u_i={\bf x}^{{\bf a}_i}$ for each $i$. Then $\partial(I)=\sum_{i=1}^mI_{n,{\bf a}_i,|{\bf a}_i|-1}$. Using \cite[Corollary 3.2]{HTaylor} (see also \cite{KM}) we obtain that
		\begin{align*}
			\reg\,\partial(I)\ &\le\ \sum_{i=1}^m\reg(I_{n,{\bf a}_i,|{\bf a}_i|-1})-(m-1)\ =\ \sum_{i=1}^m(|{\bf a}_i|-1)-m+1\\
			&=\ \sum_{i=1}^m\deg(u_i)-2m+1.
		\end{align*}
		If $I$ is a complete intersection, then $I_{n,{\bf a}_i,|{\bf a}_i|-1}$ are monomial ideals generated in pairwise disjoint sets of variables, and the above inequalities become equalities.
	\end{proof}
	
	Recall that a monomial ideal $I\subset S$ has \textit{linear quotients} if there exists an order $u_1,\dots,u_m$ of $\mathcal{G}(I)$ such that the colon ideals $(u_1,\dots,u_{i-1}):(u_i)$ are generated by variables for all $i=2,\dots,m$. By \cite[Corollary 8.2.14]{HHBook}, if $I\subset S$ has linear quotients then $\reg\,I=\omega(I)$. Hence if $I$ is generated in one degree and it has linear quotients, then it has linear resolution.
	
	In the following result we show that for a monomial ideal $I\subset S$ the difference $\reg\,I-\reg\,\partial(I)$ can be any relative integer.
	
	\begin{Theorem}\label{Thm:reg-a}
		For any relative integer $a\in\ZZ$, there exist an integer $n>0$ and a monomial ideal $I\subset S=K[x_1,\dots,x_n]$ such that
		$$
		\reg\,I-\reg\,\partial(I)\ =\ a.
		$$
	\end{Theorem}
	\begin{proof}
		Let $a\le-1$. Set $b=2-a\ge1$, and let
		$$
		I\ =\ (x_1x_2,\,x_1x_3^b,\,x_2x_4^b)\subset S=K[x_1,x_2,x_3,x_4].
		$$
		Notice that $(x_1x_2):(x_1x_3^b)=(x_2)$ and $(x_1x_2,x_1x_3^b):(x_2x_4^b)=(x_1)$. Hence, $I$ has linear quotients and $\reg\,I=\omega(I)=b+1$. Notice that $\partial(I)=(x_1,x_2,x_3^b,x_4^b)$ is a complete intersection. So $\reg\,\partial(I)=2b-1$ and $\reg\,I-\reg\,\partial(I)=2-b=a\le-1$.
		
		For $a=0$, let
		$$
		I\ =\ (x_1x_2x_3,\,x_2x_3x_4,\,x_3x_4x_5,\,x_4x_5x_6)\subset S=K[x_1,\dots,x_6].
		$$
		
		Then $I$ has linear quotients for the generators in the given order. Hence $\reg\,I=3$. We have $\reg\,\partial(I)=3$ as well, as one can check with \textit{Macaulay2} \cite{GDS}. This is also proved in the next Theorem \ref{Thm:linRes-Partial}. Hence $\reg\,I-\reg\,\partial(I)=a=0$.
		
		Now let $a\ge1$. Put $x=x_1$ and $y=x_2$. Let $b$ and $c$ integers $\ge2$ such that $b-c+1=a$. Then $b\ge c\ge2$. We put
		$$
		I\ =\ (x^c,\,x^{c-1}y,\,\dots,\,x^{c-i}y^{i},\,\dots,\,xy^{c-1},\,y^b)\subset S=K[x,y].
		$$
		
		Using \cite[Lemma 3.3]{F25} we have
		\begin{align*}
			\reg\,I\ &=\ \max\Bigg\{\substack{\displaystyle\max\{\deg(x^{c-i}y^i),\,\deg(y^b):0\le i\le c-1\}\\[3pt] \displaystyle\max\{\deg(x^{c-j}y^{j+1})-1,\,\deg(xy^{b})-1:0\le j\le c-2\}}\Bigg\}\\[3pt]
			&=\ \max\{\max\{c,\,b\},\,\max\{c+1-1,\,b+1-1\}\}\ =\ b.
		\end{align*}
		
		Since $b\ge c$, we have $\partial(I)=(x,y)^{c-1}$ and so $\reg\,\partial(I)=c-1$. Finally, we have $\reg\,I-\reg\,\partial(I)=b-(c-1)=a$, as desired.
	\end{proof}
	
	For a non-empty subset $F$ of $[n]$ we put ${\bf x}_F=\prod_{i\in F}x_i$ and ${\bf x}_\emptyset=1$.\smallskip
	
	Let $I\subset S$ be a monomial ideal with linear resolution, then
	$$
	\reg\,\partial(I)\ \ge\ \alpha(\partial(I))\ =\ \alpha(I)-1\ =\ \reg\,I-1
	$$
	and so the inequality $\reg\,\partial(I)-\reg\,I\ge-1$ holds always. On the other hand, the inequality $\reg\,\partial(I)\le\reg\,I=\alpha(I)$ needs not to hold even in such a situation. In fact, for any integer $a\ge-1$, we can find a monomial ideal $I\subset S$ with linear resolution such that $\reg\,\partial(I)-\reg\,I=a$, as we show next.
	\begin{Theorem}\label{Thm:linRes-Partial}
		For any integer $a\ge-1$, there exist $n>0$ and a monomial ideal $I\subset S=K[x_1,\dots,x_n]$ with linear resolution such that $$\reg\,\partial(I)-\reg\,I\ =\  a.$$
	\end{Theorem}
	\begin{proof}
		If $a=-1$, it is enough to consider any monomial ideal $I$ with $2$-linear resolution. Then $\partial(I)$ is a monomial prime ideal and so $\reg\,\partial(I)=1$.
		
		Now, let $a\ge0$. Put $d=a+3$ and let
		$$
		I\ =\ \sum_{i=1}^{d+1}(\prod_{j=0}^{d-1}x_{i+j})\ =\ (\prod_{j=0}^{d-1}x_{i+j}:\ i=1,\dots,d+1)\subset S=K[x_1,\dots,x_{2d}].
		$$
		
		We claim that $\reg\,I=d$ and $\reg\,\partial(I)=2d-3$. Then $\reg\,\partial(I)-\reg\,I=d-3=a$ and the statement follows.
		
		To prove that $\reg\,I=d$, we put $u_i=x_{i}x_{i+1}\cdots x_{i+d-1}$ for all $i=1,\dots,d+1$. Then $I=(u_1,\dots,u_{d+1})$ and $(u_1,\dots,u_{i-1}):(u_i)=(x_{i-1})$ for $i=2,\dots,d+1$. Hence $I$ has linear quotients and so $\reg\,I=\alpha(I)=d$.
		
		To prove that $\reg\,\partial(I)=2d-3$, we use Hochster's formula. Let $\Delta$ be the simplicial complex on vertex set $[n]$ with Stanley-Reisner ideal $I_\Delta=\partial(I)$. From Hochster's formula (\cite[Theorem 8.1.1]{HHBook}), we have
		\begin{equation}\label{eq:Hochster}
			\beta_{i,j}(S/\partial(I))\ =\ \sum_{\substack{W\subset[n]\\ |W|=j}}\dim_K\widetilde{H}_{j-i-1}(\Delta_W;K),
		\end{equation}
		where $\Delta_W$ is the restriction of $\Delta$ to $W$, and $\widetilde{H}_s(\Gamma;K)=\Ker(\partial_s)/\operatorname{Im}(\partial_{s+1})$ is the $s$-th reduced simplicial homology of the simplicial complex $\Gamma$ computed over $K$. Here $\partial_s:C_s(\Delta;K)\rightarrow C_{s-1}(\Delta;K)$ is the $K$-linear map defined as follows: $C_s(\Delta;K)$ is the $K$-vector space with basis given by the symbols ${\bf e}_F$ with $F\in\Delta$ of dimension $s-1$, and $\partial_s$ is defined by setting
		$$
		\partial_s({\bf e}_F)\ =\ \sum_{p\in F}(-1)^{\sgn(p;F)}{\bf e}_{F\setminus\{p\}},
		$$
		with $\sgn(p;F)=|\{q\in F:\ q<p\}|$. See also \cite[Section 5.1.4]{HHBook}.\smallskip
		
		First we show that $\reg\,\partial(I)\le 2d-3$. Suppose otherwise that $\reg\,\partial(I)\ge2d-2$. Since $\reg\,S/I=\reg\,I-1$, then there would exist integers $j,i$ such that $j-i\ge2d-3$ and $\widetilde{H}_{j-i-1}(\Delta_W;K)\ne0$ for some subset $W\subset[n]$ with $|W|=j$. This implies that $\dim(\Delta)\ge2d-3$. We show that this is impossible. To this end, let $H\subset[2d]$ be any subset with $|H|=2d-2$. We show that $H\notin\Delta$. This will give the desired contradiction. We can write $H=[2d]\setminus\{p,q\}$ with $p<q$. If $p\ge d$ then $H$ contains $H'=\{1,\dots,d-1\}$, and since ${\bf x}_{H'}=x_1x_2\cdots x_{d-1}=(x_1x_2\cdots x_d)/x_d\in\partial(I)$, it follows that $H'\notin\Delta$, and consequently $H\notin\Delta$. Similarly, if $q\le d+1$ then $H$ contains $H''=\{d+2,d+3,\dots,2d\}\notin\Delta$ and so $H\notin\Delta$. We may therefore assume that $p\le d-1$ and $q\ge d+2$. Hence $d,d+1\in H$. But then $H$ contains $H'''=\{1,\dots,d\}\setminus\{p\}$ and ${\bf x}_{H'''}=(x_1\cdots x_d)/x_p\in\partial(I)$. Therefore $H\notin\Delta$ and this proves our claim.
		
		Hence, $\reg\,\partial(I)\le 2d-3$. We claim that $\beta_{2,2d-2}(S/\partial(I))\ne0$. This will imply that $\reg\,\partial(I)=2d-3$ and conclude the proof.
		Let
		$$
		W=\{1,2,\dots,d-1,d+2,d+3,\dots,2d\}.
		$$
		
		We claim that $\widetilde{H}_{|W|-3}(\Delta_W;K)=\widetilde{H}_{2d-5}(\Gamma;K)\ne0$ with $\Gamma=\Delta_W$. Using (\ref{eq:Hochster}) this will show that $\beta_{2,2d-2}(S/\partial(I))\ne0$ and conclude the proof.
		
		We have $\widetilde{H}_{2d-5}(\Gamma;K)=\Ker(\partial_{2d-5})/\operatorname{Im}(\partial_{2d-4})$ where $\partial_{s}:C_{s}(\Gamma;K)\rightarrow C_{s-1}(\Gamma;K)$ is defined as before. We claim that $C_{2d-4}(\Gamma;K)=0$. Let $H$ be any subset of $W$ of size $2d-3$. Then $H=W\setminus\{p\}$ for some $p$. If $p\in\{1,2,\dots,d-1\}$, then $H$ contains $H'=\{d+2,d+3\dots,2d\}\notin\Delta$, and so $H\notin\Gamma$. Otherwise $p\in\{d+2,d+3,\dots,2d\}$ and by symmetry $H\notin\Gamma$. Therefore, $\widetilde{H}_{2d-5}(\Gamma;K)=\Ker(\partial_{2d-5})$.
		
		Next, notice that the elements of the form
		$$
		W\setminus\{p,q\}\ =\ (\{1,2,\dots,d-1\}\setminus\{p\})\cup(\{d+2,d+3,\dots,2d\}\setminus\{q\})
		$$
		with $p\in\{1,\dots,d-1\}$ and $q\in\{d+2,\dots,2d\}$ are all in $\Gamma$ and are of dimension $2d-5$. Suppose otherwise that $W\setminus\{p,q\}\notin\Delta$, then ${\bf x}_{W\setminus\{p,q\}}\in\partial(I)$. This implies that $W\setminus\{p,q\}$ should contain a set of the form $\{r,r+1,\dots,r+(d-1)\}\setminus\{s\}$ for some $1\le r\le d+1$ and some $s=r+h$ with $0\le h\le d-1$. However this is easily seen to be impossible. It follows that $W\setminus\{p,q\}\in\Delta$.
		
		Next, we claim that the element
		$$
		z\ =\ \sum_{\substack{1\le p\le d-1\\[2pt] d+2\le q\le 2d}}(-1)^{p+q}\,{\bf e}_{W\setminus\{p,q\}}\in C_{2d-5}(\Gamma;K)
		$$
		is a cycle, that is $\partial_{2d-5}(z)=0$. Then $\Ker(\partial_{2d-5})\ne0$ and this will conclude the proof. To this end, we can write
		$$
		\partial_{2d-5}(z)\ =\ \sum_{\substack{1\le p<q\le d-1\\[2pt] d+2\le r\le 2d\\}}c_{pqr}\,{\bf e}_{W\setminus\{p,q,r\}}+\sum_{\substack{1\le p\le d-1\\[2pt] d+2\le q<r\le 2d\\}}c_{pqr}\,{\bf e}_{W\setminus\{p,q,r\}},
		$$
		with $c_{pqr}\in K$.
		
		We must prove that $c_{pqr}=0$ for all $p,q,r$. There are two cases to consider. We verify only one, and leave the other to the reader. Let $\{p<q<r\}\subset W$ with $p,q\in\{1,2,\dots,d-1\}$ and $r\in\{d+2,d+3,\dots,2d\}$. Then, in $z$ there are only two summands $(-1)^{p+r}{\bf e}_{W\setminus\{p,r\}}$ and $(-1)^{q+r}{\bf e}_{W\setminus\{q,r\}}$ whose derivation contributes to the sign of ${\bf e}_{W\setminus\{p,q,r\}}$ in the expansion of $\partial_{2d-5}(z)$. Since $p<q$, we have
		\begin{align*}
			c_{pqr}\ &=\ (-1)^{p+r}(-1)^{\sgn(W\setminus\{p,r\};q)}+(-1)^{q+r}(-1)^{\sgn(W\setminus\{q,r\};p)}\\
			&=\ (-1)^{p+r}(-1)^{|\{s\in W\setminus\{p,r\}\,:\,s<q\}|}+(-1)^{q+r}(-1)^{|\{s\in W\setminus\{q,r\}\,:\,s<p\}|}\\
			&=\ (-1)^{p+r}(-1)^{q-2}+(-1)^{q+r}(-1)^{p-1}\\
			&=\ (-1)^{p+q+r}(1-1)\ =\ 0,
		\end{align*}
		as desired.
	\end{proof}
	
	On the other hand, in support of the normal behaviour we have:
	\begin{Corollary}
		Let $I\subset S$ be a monomial ideal. Then $\partial(\m^kI)$ has linear quotients for all $k\gg0$. In particular, $\reg\,\partial(\m^kI)<\reg(\m^kI)$ for all $k\gg0$.
	\end{Corollary}
	\begin{proof}
		It is clear that $\partial(\m^k)=\m^{k-1}$ and $I\subset\m\,\partial(I)$ for all integers $k\ge1$. Hence $\m^{k-1}I\subset\m^k\partial(I)$. Using Leibniz's rule which is valid for the derivation of the product of monomial ideals, we have
		$$
		\partial(\m^kI)\ =\ \partial(\m^k)I+\m^k\partial(I)\ =\ \m^{k-1}I+\m^k\partial(I)\ =\ \m^k\partial(I)
		$$
		for all $k\ge1$. It is proved in \cite[Theorem 1.6]{FHM} that $\m^kJ$ has linear quotients for any monomial ideal $J\subset S$ and all $k\gg0$. Hence $\partial(\m^kI)=\m^k\partial(I)$ has linear quotients for all $k\gg0$. For a monomial ideal $J\subset S$ with linear quotients its regularity is $\omega(J)$. Since we obviously have $\omega(\partial(\m^kI))<\omega(\m^kI)$, we conclude that $\reg\,\partial(\m^kI)<\reg(\m^kI)$ for all $k\gg0$.
	\end{proof}
	
	\section{Classes of ideals closed under the gradient operation}
	
	In this section, we present several families $\mathcal{C}$ of monomial ideals which are closed under the gradient operation: if $I\in\mathcal{C}$ then $\partial(I)\in\mathcal{C}$ as well.
	
	Recall that a monomial ideal $I\subset S$ is called \textit{polymatroidal} if the exponent vectors of the minimal monomial generators of $I$ form the set of bases of a discrete polymatroid, see \cite[Section 12.6]{HHBook}. Polymatroidal ideals are generated in one degree and they have linear quotients. Hence, they have linear resolution.
	
	\begin{Theorem}\label{Thm:partial-polym}
		Let $I\subset S$ be a polymatroidal ideal. Then $\partial(I)$ is polymatroidal.
	\end{Theorem}
	\begin{proof}
		Let $P$ be the discrete polymatroid on $[n]$ with $B(P)=\{{\bf u}:{\bf x^u}\in\mathcal{G}(I)\}$ as the set of bases of $P$. We have $\rank\,P=\alpha(I)=d$. Let $Q=\{{\bf u}\in P:|{\bf u}|\le d-1\}$. By \cite[Lemma 12.2.3.(a)]{HHBook}, $Q$ is a discrete polymatroid as well. Finally, it is clear that $B(Q)=\{{\bf u}:{\bf x^u}\in\mathcal{G}(\partial(I))\}$, and this concludes the proof.
	\end{proof}
	
	Next, we want to extend Theorem \ref{Thm:partial-polym} to a wider class of ideals. For a homogeneous ideal $I\subset S$ and an integer $j\ge0$, we let $I_{\langle j\rangle}$ be the ideal generated by all homogeneous polynomials of degree $j$ belonging to $I$. The following technical lemma is required.
	
	\begin{Lemma}\label{Lem:component}
		Let $I\subset S$ be a monomial ideal. Then $\partial(I)_{\langle j\rangle}=\partial(I_{\langle j+1\rangle})$ for all $j$.
	\end{Lemma}
	\begin{proof}
		As noticed before, we can write $\partial(I)=\sum_{i\in[n]}(I:x_i)$. Then
		$$
		\partial(I)_{\langle j\rangle}\ =\ (\sum_{i\in[n]}(I:x_i))_{\langle j\rangle}\ =\ \sum_{i\in[n]}(I:x_i)_{\langle j\rangle}\ =\ \sum_{i\in[n]}(I_{\langle j+1\rangle}:x_i)\ =\ \partial(I_{\langle j+1\rangle}),
		$$
		as desired.
	\end{proof}
	
	Recall that a monomial ideal $I\subset S$ (not necessarily squarefree) is called \textit{vertex splittable} \cite{MK} if it can be obtained by the following recursive procedure.
	\begin{enumerate}
		\item[(i)] If $u$ is a monomial and $I=(u)$, $I=(0)$ or $I=S$, then $I$ is vertex splittable.\smallskip
		\item[(ii)] If there exists a variable $x_i$ and vertex splittable ideals $I_1\subset S$ and $I_2\subset K[x_1,\dots,x_{i-1},x_{i+1},\dots,x_n]$ such that $I=x_iI_1+I_2$, $I_2\subseteq I_1$ and $\mathcal{G}(I)$ is the disjoint union of $\mathcal{G}(x_iI_1)$ and $\mathcal{G}(I_2)$, then $I$ is vertex splittable.
	\end{enumerate}
	
	In the case (ii), the decomposition $I=x_iI_1+I_2$ is called a \textit{vertex splitting} of $I$ and $x_i$ is called a \textit{splitting vertex} of $I$.
	
	A vertex splittable ideal has linear quotients \cite[Theorem 2.4]{MK}. This result is proved in \cite{MK} under the assumption that $I$ is squarefree. However this assumption is unnecessary, and the proof given in \cite[Theorem 2.4]{MK} works for any vertex splittable ideal as we consider above.
	
	In the squarefree case, writing $I=I_\Delta$, Moradi and Khosh-Ahang proved that $I$ is vertex splittable if and only if the Alexander dual $\Delta^\vee$ is a vertex decomposable simplicial complex, see \cite[Theorem 2.3]{MK}.
	
	A monomial ideal $I\subset S$ is called \textit{componentwise polymatroidal} if $I_{\langle j\rangle}$ is polymatroidal for all $j\ge0$. Any polymatroidal ideal is componentwise polymatroidal. It is proved in \cite[Theorem 6]{Fic0} that componentwise discrete polymatroidal ideals have linear quotients. In fact, it is noticed in \cite[Proposition 2]{CF} that they even are vertex splittable. Combining this fact with Theorem \ref{Thm:partial-polym} and Lemma \ref{Lem:component} we obtain that
	\begin{Corollary}
		Let $I\subset S$ be a componentwise polymatroidal ideal. Then $\partial(I)$ is again componentwise polymatroidal, in particular it is vertex splittable.
	\end{Corollary}
	
	Another prominent class of monomial ideals is the class of (strongly) stable ideals. For a monomial $u\in S$ with $u\ne 1$, recall that the \textit{support} of $u$ is the set $$\supp(u)=\{i\in[n]:\ x_i\ \textup{divides}\ u\}$$ and the \textit{maximum} of $u$ is the integer $$\max(u)=\max\supp(u)=\max\{i\in[n]:\ x_i\ \textup{divides}\ u\}.$$
	
	A monomial ideal $I\subset S$ is called \textit{stable} if for all monomials $u\in I$ with $u\ne 1$, and all $i<\max(u)$, we have that $x_i(u/x_{\max(u)})\in I$. Whereas, $I$ is called \textit{strongly stable} if for all monomials $u\in I$ and all $i<j$ with $x_j$ dividing $u$, we have $x_i(u/x_j)\in I$.
	
	Any strongly stable ideal is stable, but the converse does not hold true in general. (Strongly) stable ideals have linear quotients, in fact they even are vertex splittable, see \cite[Proposition 1]{CF}. Hence $\reg\,I=\omega(I)$ for any (strongly) stable ideal $I\subset S$.
	
	\begin{Proposition}\label{Prop:stable}
		Let $I$ be a $($strongly$)$ stable monomial ideal. Then $\partial(I)$ is again $($strongly$)$ stable. In particular $\reg\,\partial(I)\le\reg\,I$.
	\end{Proposition}
	\begin{proof}
		We prove that if $I$ is stable, then $\partial(I)$ is stable as well. The proof in the case that $I$ is strongly stable is analogous. Let $u\in\partial(I)$ be a monomial and let $i<\max(u)$. We must prove that $x_i(u/x_{\max(u)})\in\partial(I)$. Since $u\in\partial(I)$, there exists $j$ such that $v=x_ju\in I$. We distinguish the two possible cases.\medskip
		
		\textbf{Case 1.} Let $j<\max(v)$. Then $i<\max(u)=\max(v)$. Since $I$ is stable, it follows that $x_i(v/x_{\max(v)})\in I$. Since $j<\max(v)$ we have that $x_j$ divides $x_i(v/x_{\max(v)})$. Hence $(x_i(v/x_{\max(v)}))/x_j=(x_i(v/x_j))/x_{\max(u)}=x_i(u/x_{\max(u)})\in\partial(I)$, as desired.\smallskip
		
		\textbf{Case 2.} Let $j=\max(v)$. Then $i<\max(u)<j=\max(v)$. Since $I$ is stable, we have that $x_i(v/x_{\max(v)})=x_i(v/x_j)=x_iu\in I$. Using that $x_{\max(u)}<j$, it follows that $x_{\max(u)}$ divides $x_iu\in I$. Hence $x_i(u/x_{\max(u)})\in\partial(I)$, as desired.
	\end{proof}

	\section{Stanley-Reisner ideals with differential linear resolution}

	For a homogeneous ideal $I\subset S$, we put $\partial^0(I)=I$, $\partial^1(I)=\partial(I)$ and for any integer $\ell>1$ we put $\partial^\ell(I)=\partial(\partial^{\ell-1}(I))$ recursively.
	
	We have seen in Theorem \ref{Thm:linRes-Partial} that even if $I\subset S$ is a monomial ideal with linear resolution, the regularity $\reg\,\partial(I)$ can exceed any given number $a\ge\reg\,I=\alpha(I)$. Motivated by this fact, we introduce the following concept.
	\begin{Definition}
		\rm We say that a homogeneous ideal $I\subset S$ has \textit{differential linear resolution} if $\partial^\ell(I)$ has linear resolution for all $\ell=0,\dots,\alpha(I)$. That is,
		$$
		\reg\,\partial^\ell(I)\ =\ \alpha(I)-\ell,
		$$
		for all $\ell=0,\dots,\alpha(I)$.
	\end{Definition}
	
	By Theorem \ref{Thm:partial-polym} and Proposition \ref{Prop:stable} we know that polymatroidal ideals and equigenerated (strongly) stable ideals have differential linear resolution.\smallskip
	
	In the squarefree case, we have the following result.
	\begin{Theorem}\label{Thm:sqfree-d}
		Let $I\subset S$ be a squarefree monomial ideal equigenerated in degree $d$. If $d\in\{0,1,2,n-2,n-1,n\}$, the following conditions are equivalent.
		\begin{enumerate}
			\item[\textup{(a)}] $I$ has linear resolution.
			\item[\textup{(b)}] $I$ has linear quotients.
			\item[\textup{(c)}] $I$ is vertex splittable.
			\item[\textup{(d)}] $I$ has differential linear resolution.
		\end{enumerate}
	\end{Theorem}
	
	The equivalence (a) $\Leftrightarrow$ (b) follows from \cite[Theorem E]{FM}. Next, if $d\in\{0,1,n-1,n\}$, then $I$ is polymatroidal and the remaining equivalences (b) $\Leftrightarrow$ (c) $\Leftrightarrow$ (d) follow from \cite[Proposition 2]{CF} and Theorem \ref{Thm:partial-polym}. When $d=2$, the equivalence (a) $\Leftrightarrow$ (c) was proved in \cite[Corollary 3.8]{MK}. Moreover, for such an ideal $I$, we have that $\partial(I)$ is a monomial prime ideal and so $\reg\,\partial(I)=1$ and $\reg\,\partial^2(I)=0$, so the equivalence (a) $\Leftrightarrow$ (d) follows in the case $d=2$. Since any equigenerated vertex splittable ideal has linear resolution, from this discussion, to complete the proof of Theorem \ref{Thm:sqfree-d} it remains to prove that statement (a) implies (c) and (d) in the case $d=n-2$. This is shown in the next
	\begin{Theorem}\label{Thm:complementary}
		Let $I\subset S$ be a squarefree monomial ideal with a $(n-2)$-linear resolution. Then $\partial^\ell(I)$ is vertex splittable and so $$\reg\,\partial^\ell(I)\ =\ n-2-\ell,$$ for all $\ell=0,\dots,n-2$.
	\end{Theorem}
	
	The proof of this result requires some preparation. The \textit{support} of a monomial ideal $I\subset S$ is the set defined as
	$$
	\supp(I)\ =\ \bigcup_{u\in\mathcal{G}(I)}\supp(u).
	$$
	\begin{Lemma}\label{Lem:part}
		Let $I\subset S$ be a monomial ideal, and let $i\notin\supp(I)$. Then
		$$
		\partial^\ell(x_iI)\ =\ x_i\partial^{\ell}(I)+\partial^{\ell-1}(I).
		$$
		If, in addition, $\partial^{\ell-1}(I)$ and $\partial^\ell(I)$ are vertex splittable, then $\partial^\ell(x_iI)$ is vertex splittable.
	\end{Lemma}
	\begin{proof}
		It is enough to prove the statement in the case $\ell=1$. Using Leibniz's rule which is valid for the derivation of the product of monomial ideals, we have that $\partial(x_iI)=x_i\partial(I)+I$. Since $i\notin\supp(I)$, $I\subset\partial(I)$ and both ideals are vertex splittable it follows that $\partial(x_iI)$ is vertex splittable too.
	\end{proof}
	
	Recall that by \cite{FM1,HQS} (see also \cite{FM2}) we can interpret any squarefree monomial ideal $I\subset S$ equigenerated in degree $n-2$ as the \textit{complementary edge ideal}
	$$
	I\ =\ I_c(G)\ =\ ({\bf x}_{[n]}/(x_ix_j):\{i,j\}\in E(G))
	$$
	of a finite simple graph $G$ on the vertex set $[n]=V(G)$ with edge set $E(G)$.
	
	\begin{proof}[Proof of Theorem \ref{Thm:complementary}]
		We can write $I=I_c(G)$ where $G$ is a finite simple graph on the vertex set $[n]$. We may assume that $G$ does not contain isolated vertices. In fact, assume that $n$ is an isolated vertex of $G$. Then $I=x_nJ$ where $J=I_c(H)$ has $(n-3)$-linear resolution and $H=G\setminus\{n\}$ is the graph on vertex set $[n-1]$. If the statement holds for $J$, by Lemma \ref{Lem:part} it holds for $I$ as well. Therefore, we may assume that $G$ does not contain isolated vertices. By \cite[Theorem B]{FM} it follows that $G$ is connected. Let $n,n-1,\dots,1$ be a labeling of the vertices of $G$ such that $G\setminus\{n,n-1\dots,n-i\}$ is connected for all $i=1,\dots,n$. Then
		$$
		I\ =\ x_nJ+L
		$$
		where $J=I_c(G_n)$ and $L=({\bf x}_{[n]}/(x_ix_n):\ \{i,n\}\in E(G)).$
		
		Since $G_n$ is connected on $n-1$ vertices, by induction $I_c(G_n)$ is vertex splittable. Notice that since $L$ can be regarded as a squarefree monomial ideal in $K[x_1,\dots,x_{n-1}]$ generated in degree $n-2$, it is polymatroidal and therefore it is vertex splittable. We claim that $L\subset J$. In fact, let $u={\bf x}_{[n]}/(x_ix_n)\in\mathcal{G}(L)$. Since $G_n$ is connected, we can find an edge $\{i,j\}\in E(G_n)$. Then ${\bf x}_{[n-1]}/(x_ix_j)\in I_c(G_n)$ divides $u$. Hence $u\in I_c(G_n)$ and so $L\subset I_c(G_n)$. We conclude that $\partial^0(I)=I=x_nI_c(G_n)+L$ is in fact vertex splittable.
		
		We have
		$$
		\partial(I)\ =\ \partial(x_1J+L)\ =\ x_1\partial(J)+(J+\partial(L)).
		$$
		
		Let $H$ be the graph on vertex set $[n-1]$ with edge set
		$$
		E(H)\ =\ E(G_n)\cup\{\{i,j\}:\ \{i,n\}\in E(G),\ j\in[n-1]\setminus\{i\}\}.
		$$
		
		We claim that
		\begin{enumerate}
			\item[(a)] $J+\partial(L)=I_c(H)$, and
			\item[(b)] $H$ is connected.
		\end{enumerate}
		
		Since $G_n$ is connected and $E(H)$ contains $E(G_n)$ it follows that $H$ is connected too and statement (b) holds. Statement (a) follows because
		\begin{align*}
			\mathcal{G}(J&+\partial(L))=\mathcal{G}(J)\cup\mathcal{G}(\partial(L))\\
			&=\{{\bf x}_{[n-1]\setminus\{i,j\}}:\{i,j\}\in E(G_n)\}\cup\{{\bf x}_{[n]\setminus\{i,n\}}/x_j:\{i,n\}\in E(G),j\ne i,j\ne n\}\\
			&=\{{\bf x}_{[n-1]\setminus\{i,j\}}:\{i,j\}\in E(G_n)\}\cup\{{\bf x}_{[n-1]\setminus\{i,j\}}:\{i,n\}\in E(G),j\in[n-1]\setminus\{i\}\}\\
			&=\{{\bf x}_{[n-1]\setminus\{p,q\}}:\{p,q\}\in E(H)\}=\mathcal{G}(I_c(H)).
		\end{align*}
		
		Now, let $\ell\in[n-2]$. A quick calculation shows that
		$$
		\partial^\ell(I)\ =\ x_1\partial^{\ell}(J)+(\partial^{\ell-1}(J)+\partial^\ell(L)).
		$$
		
		Since we have $I_c(H)=J+\partial(L)$, and so $\partial^{\ell-1}(I_c(H))=\partial^{\ell-1}(J)+\partial^\ell(J)$, the previous formula becomes
		\begin{equation}\label{eq:vs-ind}
			\partial^\ell(I)\ =\ x_1\partial^{\ell}(J)+\partial^{\ell-1}(I_c(H)).
		\end{equation}
		
		Now, we claim that $\partial^{\ell-1}(I_c(H))\subset\partial^\ell(J)$. Since $\partial^{\ell-1}(I_c(H))=\partial^{\ell-1}(J)+\partial^\ell(L)$ and obviously $\partial^{\ell-1}(J)\subset\partial^\ell(J)$, it remains to prove that $\partial^\ell(L)$ is contained in $\partial^{\ell}(J)$.
		
		Let $u={\bf x}_{[n-1]}/(x_px_{q_1}\cdots x_{q_\ell})\in\mathcal{G}(\partial^\ell(L))$ be a squarefree monomial generator with $\{p,n\}\in E(G)$. Since $G_n$ is connected, there exists an edge $\{p,r\}\in E(G_n)$. We distinguish the two possible cases.\smallskip
		
		\textbf{Case 1.} Let $r=q_i$ for some $i$. Say $i=1$. Then $v={\bf x}_{[n-1]}/(x_px_{q_1})\in I_c(H)$ and consequently $u=v/(x_{q_2}\cdots x_{q_{\ell}})\in\partial^{\ell-1}(J)\subset\partial^\ell(J)$, as desired.\smallskip
		
		\textbf{Case 2.} Let $r\ne q_i$ for all $i$. Then $u/x_r=({\bf x}_{[n-1]}/(x_px_r))/(x_{q_1}\cdots x_{q_\ell})\in\partial^\ell(J)$. This shows that $u\in\partial^\ell(J)$, as desired.\smallskip
		
		Hence $\partial^{\ell-1}(I_c(H))\subset\partial^\ell(J)$. By induction on $|V(G)|$, $J=I_c(G_n)$ and $I_c(H)$ satisfy the statement because the graphs $G_n$ and $H$ are connected on a smaller number of vertices than $G$. Therefore $\partial^\ell(J)$ and $\partial^{\ell-1}(I_c(H))$ are vertex splittable. Equation (\ref{eq:vs-ind}) implies that $\partial^\ell(I)$ is indeed vertex splittable.
	\end{proof}
	
	\section{Squarefree monomial ideals with many generators}
	
	Let $\mu(I)$ be the minimal number of generators of a homogeneous ideal $I\subset S$. If $I\subset S$ is a monomial ideal, then $\mu(I)=|\mathcal{G}(I)|$.
	
	Let $I\subset S$ be a squarefree monomial ideal generated in degree $d$. Then $\mu(I)\le\binom{n}{d}$. It was proved by Terai and Yoshida \cite[Theorem 3.3]{TY} (see also \cite[Corollary 4.6]{DNSVW}) that if $\mu(I)\ge\binom{n}{d}-2d+1$ then $I$ has linear resolution. In fact, inspecting the proof given in \cite[Corollary 4.6]{DNSVW}, if we write $I=I_\Delta$ and $\mu(I)\ge\binom{n}{d}-2d+1$ the authors prove that $\Delta^\vee$ is vertex decomposable. Hence, by \cite[Theorem 2.3]{MK} $I$ is vertex splittable.
	
	Let $I\subset S$ be a squarefree monomial ideal equigenerated in degree $d$ such that $\mu(I)\ge\binom{n}{d}-2d+1$. Is it true that $\partial^\ell(I)$ is vertex splittable, and thus has linear resolution, for all $0\le\ell\le d$ ? At present we are not able yet to answer this question in full generality. On the other hand, applying the classical Kruskal-Katona's theorem \cite{Katona,Kruskal}, we can show the following partial result.
	\begin{Theorem}\label{Thm:Kruskal}
		Let $I\subset S$ be a squarefree monomial ideal generated in degree $d$. Assume that $\mu(I)\ge\binom{n}{d}-2d+1$ and $n\ge 2d$. Then $\partial^\ell(I)$ is vertex splittable and $\reg\,\partial^\ell(I)=d-\ell$, for all $0\le\ell\le d$.
	\end{Theorem}
	\begin{proof}
		We may assume that $d\ge3$ otherwise the statement holds vacuously. 
		
		Let $\mathcal{M}=\mathcal{G}(I)$. Recall that the \textit{shadow} of $\mathcal{M}$ is defined as the set
		$$
		\partial(\mathcal{M})\ =\ \{u/x_i:\ u\in\mathcal{M},\,i\in\supp(u)\}.
		$$ 
		
		Then $\partial(I)$ is the ideal generated by $\partial(\mathcal{M})$. Let $|\mathcal{M}|=a=\binom{a_d}{d}+\binom{a_{d-1}}{d-1}+\dots+\binom{a_t}{t}$ be the $d$-th binomial expansion of $|\mathcal{M}|$. This expression is uniquely determined by the inequalities $a_d>a_{d-1}>\dots>a_t\ge t$, see \cite[Lemma 6.3.4]{HHBook}.
		
		By the Kruskal-Katona's theorem (\cite{Katona,Kruskal}) we have
		$$
		|\partial(\mathcal{M})|\ \ge\ a^{(d-1)}\ =\ \binom{a_d}{d-1}+\binom{a_{d-1}}{d-2}+\dots+\binom{a_t}{t-1}.
		$$
		
		If we show that $a^{(d-1)}\ge\binom{n}{d-1}-2(d-1)+1$, then applying again \cite[Theorem 3.3]{TY} (or \cite[Corollary 4.6]{DNSVW}), it follows that $\partial(I)$ is vertex splittable. Since the assumptions are again satisfied for $\partial(I)$, it is enough to prove the statement for $I=\partial^0(I)$. Hence, it remains to prove that $a^{(d-1)}\ge\binom{n}{d-1}-2(d-1)+1$.
		
		The operator $b\mapsto b^{(d-1)}$ is an increasing function. Therefore, we may also assume that $|\mathcal{M}|=a=\binom{n}{d}-2d+1$ and it is enough to show that
		$$
		a^{(d-1)}\ \ge\ \binom{n}{d-1}-2(d-1)+1.
		$$
		
		To this end, we claim that the $d$-th binomial expansion of $a$ is:
		\begin{equation}\label{eq:Macaulay}
			a\ =\ \sum_{i=1}^{d-2}\binom{n-i}{d-i+1}+\begin{cases}
				\binom{n-d}{2}+\binom{2n-4d+2}{1}&\textup{if}\ 2d\le n\le 3d-3,\\[3pt]
				\binom{n-d+1}{2}&\textup{if}\ n=3d-2,\\[3pt]
				\binom{n-d+1}{2}+\binom{n-3d+2}{1}&\textup{if}\ n\ge 3d-1.
			\end{cases}
		\end{equation}
		
		Using Pascal's rule, we have $\binom{n}{d}=\binom{n-1}{d}+\binom{n-1}{d-1}$. Hence,
		$$
		a\ =\ \binom{n}{d}-2d+1\ =\ \binom{n-1}{d}+\binom{n-1}{d-1}-2d+1.
		$$
		
		Applying again Pascal's rule to $\binom{n-1}{d-1}$ and iterating this procedure yields that
		\begin{equation}\label{eq:a}
			a\ =\ \binom{n-1}{d}+\binom{n-2}{d-1}+\dots+\binom{n-d+2}{3}+\binom{n-d+2}{2}-2d+1.
		\end{equation}
		
		Since $n\ge 2d$, we have $\binom{n-d+2}{2}\ge\binom{d+2}{2}=(d+2)(d+1)/2>2d-1$. Hence $b=\binom{n-d+2}{2}-2d+1>0$. Now we distinguish two cases.\medskip
		
		\textbf{Case 1.} Let $2d\le n\le 3d-3$. Notice that
		\begin{equation}\label{eq:b}
			\begin{aligned}
				b-\binom{n-d}{2}\ &=\ \binom{n-d+2}{2}-\binom{n-d}{2}-2d+1\\
				&=\ \frac{(n-d+2)(n-d+1)-(n-d)(n-d-1)-4d+2}{2}\\ 
				&=\ \frac{4n-8d+4}{2}\ = \binom{2n-4d+2}{1}\ \ge\ 1,
			\end{aligned}
		\end{equation}
		in the given range of values of $n$. Combining (\ref{eq:b}) with (\ref{eq:a}) we obtain the expression given in the first case of equation (\ref{eq:Macaulay}). Such expression is indeed the $d$-th binomial expansion of $a$, because
		$$
		n-1>n-2>\dots>n-d+2>n-d>2n-4d+2\ge1,
		$$
		where the second to last inequality follows because $n\le 3d-3$.\smallskip
		
		\textbf{Case 2.} Let $n\ge 3d-2$. Using Pascal's rule we have
		\begin{equation}\label{eq:b1}
			\begin{aligned}
				b-\binom{n-d+1}{2}\ &=\ \binom{n-d+2}{2}-\binom{n-d+1}{2}-2d+1\\
				&=\ \binom{n-d+1}{1}-2d+1\ =\ n-3d+2\ \ge\ 0.
			\end{aligned}
		\end{equation}
		Combining equation (\ref{eq:b1}) with (\ref{eq:a}), and distinguish between the cases $n=3d-2$ and $n\ge3d-1$, we obtain the expressions given in the second and third cases of equation (\ref{eq:Macaulay}). These expressions are indeed the $d$-th binomial expansion of $a$, because
		$$
		n-1>n-2>\dots>n-d+2>n-d+1\ge 2,
		$$
		and in addition we have $n-d+1>n-3d+2\ge1$ if $n\ge 3d-1$.\medskip
		
		Having acquired (\ref{eq:Macaulay}), using that $\binom{N}{0}=1$ and $\binom{N}{1}=N$ for all $N\ge0$, we obtain that
		\begin{equation}\label{eq:ad-1}
		a^{(d-1)}\ =\ \sum_{i=1}^{d-2}\binom{n-i}{d-i}+\begin{cases}
			\,n-d+1&\textup{if}\ 2d\le n\le 3d-2,\\
			\,n-d+2&\textup{if}\ n\ge 3d-1.
		\end{cases}
		\end{equation}
		
		Iterations of Pascal's rule yield that $\binom{n}{d-1}=\sum_{j=1}^{d}\binom{n-j}{d-j}$. Therefore,
		\begin{align*}
			\sum_{i=1}^{d-2}\binom{n-i}{d-i}\ &=\ \binom{n}{d-1}-\binom{n-(d-1)}{1}-\binom{n-d}{0}\\ &=\ \binom{n}{d-1}-n+d-2.
		\end{align*}
		
		Combining this fact with (\ref{eq:ad-1}) we obtain that
		\begin{align*}
			a^{(d-1)}\ &=\ \begin{cases}
				\binom{n}{d-1}-1&\textup{if}\ 2d\le n\le 3d-2,\\[3pt]
				\binom{n}{d-1}&\textup{if}\ n\ge 3d-1.
			\end{cases}
		\end{align*}
		Hence, $a^{(d-1)}\ge\binom{n}{d-1}-2(d-1)+1$, as desired.
	\end{proof}
	
	For the proof of the previous theorem to hold, the hypothesis $n\ge2d$ cannot be dropped. In fact, let $n=20$ and $d=17$, then $n<2d$. We have
	\begin{align*}
		\textstyle a&=\textstyle \binom{n}{d}-2d+1=\binom{20}{17}-2\cdot17+1=1107,\ \ \text{and}\\
		\textstyle b&=\textstyle \binom{n}{d-1}-2(d-1)+1=\binom{20}{16}-2\cdot16+1=4814.
	\end{align*}
	However $a^{(d-1)}=1107^{(16)}$ is ``only" $4813$ which is strictly less than $b=4814$.\medskip
	
	Therefore, if the conclusion of Theorem \ref{Thm:Kruskal} holds regardless of the hypothesis ``$n\ge2d$" a different argument would be required.
	
	\section{Powers of squarefree quadratic monomial ideals}
	
	It is proved in \cite{HHZ} (see also \cite[Theorem 1.1]{FShort} or \cite[Theorem E]{FM}) that if $I\subset S$ is a squarefree monomial ideal with $2$-linear resolution, then $I^k$ has $2k$-linear resolution as well, for all $k\ge1$.
	
	In this final section, we prove the stronger fact that if $I$ has $2$-linear resolution then $I^k$ has differential linear resolution for all $k\ge1$.
	
	\begin{Theorem}\label{Thm:edgeideal}
		Let $I\subset S$ be a squarefree monomial ideal with $2$-linear resolution. Then $\partial^\ell(I^k)$ has linear quotients, for all $k\ge1$ and all $0\le\ell\le 2k$. In particular,
		$$
		\reg\,\partial^\ell(I^k)\ =\ 2k-\ell,
		$$
		for all $k\ge1$ and all $0\le\ell\le 2k$.
	\end{Theorem}
	\begin{proof}
		For the proof of this result, we use some results from \cite{HHZ} and \cite{FShort}.
		
		Let $I\subset S$ be a quadratic squarefree monomial ideal with linear resolution. Without loss of generality, we can assume that $I$ is \textit{fully-supported}, that is $\supp(I)=[n]$. Up to a suitable relabeling on the variables, by \cite[Lemma 2.4 and Corollary 3.1]{FShort} we can write $I=x_nP+J$, for monomial ideals $P,J\subset S$ satisfying the properties:
		\begin{enumerate}
			\item[(i)] $n\notin\supp(P)\cup\supp(J)$,
			\item[(ii)] $P$ is a monomial prime ideal and $J\subset P$,
			\item[(iii)] $P^rJ^s$ has linear quotients, and so linear resolution, for all $r,s\ge0$.
		\end{enumerate}
		
		Now, let $k\ge1$. Then $I^k=\sum_{r=0}^kx_n^{k-r}P^{k-r}J^r$. If $\ell=0$, then $\partial^0(I^k)=I^k$ has linear quotients by \cite{HHZ} (see also \cite[Theorem 1.1]{FShort}).\smallskip
		
		Next, let $\ell\ge1$. We claim that
		\begin{equation}\label{eq:partialI(G)}
			\partial^\ell(I^k)\ =\ \begin{cases}
				\displaystyle\,(\sum_{r=0}^\ell x_n^{k-r}\n^rP^{k-\ell})+(\sum_{s=\ell+1}^kx_n^{k-s}\n^\ell P^{k-s}J^{s-\ell})&\textup{for}\ 1\le\ell\le k-1,\\[2pt]
				\,\,\,\m^{2k-\ell}&\textup{for}\ k\le\ell\le 2k,
			\end{cases}
		\end{equation}
		where $\n=(x_1,\dots,x_{n-1})$ and $\m=(\n,x_n)=(x_1,\dots,x_{n-1},x_{n})$.
		
		For $\ell=1$, we have
		\begin{equation}\label{eq:HHZ1}
			\begin{aligned}
				\partial(I^k)\ &=\ \partial(\sum_{s=0}^kx_n^{k-s}P^{k-s}J^s)\ =\ \sum_{s=0}^k\partial(x_n^{k-s}P^{k-s}J^s)\\
				&=\ \sum_{s=0}^k\partial(x_n^{k-s})P^{k-s}J^s+\sum_{s=0}^kx_n^{k-s}\partial(P^{k-s}J^{s}).
			\end{aligned}
		\end{equation}
		
		Next, we compute both sums appearing in the above expression. We have
		\begin{equation}\label{eq:HHZ2}
			\sum_{s=0}^k\partial(x_n^{k-s})P^{k-s}J^s= \sum_{s=0}^{k-1}x_{n}^{k-s-1}P^{k-s}J^s= \sum_{s=1}^kx_n^{k-s}P^{k+1-s}J^{s-1}.
		\end{equation}
		
		For the second sum, we use that $\partial(P^{r})=P^{r-1}$ for $r\ge1$. Hence
		\begin{equation}\label{eq:HHZ3}
			\begin{aligned}
				\sum_{s=0}^kx_n^{k-s}\partial(P^{k-s}J^s)\ &=\ \sum_{s=0}^kx_n^{k-s}[\partial(P^{k-s})J^s+P^{k-s}\partial(J^s)]\\
				&=\ \sum_{s=0}^kx_n^{k-s}[P^{k-s-1}J^r+P^{k-s}J^{s-1}\partial(J)].
			\end{aligned}
		\end{equation}
		
		Combining equations (\ref{eq:HHZ1}), (\ref{eq:HHZ2}) and (\ref{eq:HHZ3}) we obtain
		\begin{equation}
			\begin{aligned}\label{eq:HHZ4}
				\partial(I^k)\ &=\ x_n^kP^{k-1}+\sum_{s=1}^kx_n^{k-s}[P^{k+1-s}J^{s-1}+P^{k-s-1}J^s+P^{k-s}J^{s-1}\partial(J)]\\
				&=\ x_n^kP^{k-1}+\sum_{s=1}^kx_n^{k-s}P^{k-s-1}J^{s-1}[P^{2}+J+P\partial(J)].
			\end{aligned}
		\end{equation}
		
		We claim that
		\begin{equation}\label{eq:HHZ5}
			P^2+J+P\partial(J)\ =\ P(P+\partial(J)).
		\end{equation}
		
		Clearly $P(P+\partial(J))\subset P^2+J+P\partial(J)$. Conversely, let $u\in\mathcal{G}(P^2+J+P\partial(J))$. If $u\in P^2+P\partial(J)$ then $u\in\mathcal{G}(P(P+\partial(J)))$. Else, if $u\in\mathcal{G}(J)$, using (ii) we have $u=x_px_q$ with $x_p\in P$. But $x_q\in\partial(J)$, hence $u\in\mathcal{G}(P\partial(J))\subset\mathcal{G}(P(P+\partial(J)))$, as desired.
		
		Since $I$ is fully-supported, and $n\notin\supp(P)\cup\supp(J)$ by property (i), we obtain that $P+\partial(J)=\n=(x_1,\dots,x_{n-1})$. Combining equations (\ref{eq:HHZ4}) and (\ref{eq:HHZ5}) we obtain
		\begin{equation}\label{eq:HHZ6}
			\partial(I^k)\ =\ x_{n}^{k}P^{k-1}+x_{n}^{k-1}\n P^{k-1}+\sum_{s=2}^kx_n^{k-s}\n P^{k-s}J^{s-1}.
		\end{equation}
	    This is the expression given in equation (\ref{eq:partialI(G)}) for $\ell=1$.
	    
	    Now, let $1<\ell\le k-1$, and assume by finite induction that the formula for $\partial^\ell(I^k)$ given in equation (\ref{eq:partialI(G)}) is correct. Then,
	    \begin{equation}\label{eq:partial-ell}
	    	\partial^{\ell+1}(I^k)\ =\ \partial(\partial^\ell(I^k))\ =\ \partial((\sum_{r=0}^\ell x_n^{k-r}\n^rP^{k-\ell})+(\sum_{s=\ell+1}^kx_n^{k-s}\n^\ell P^{k-s}J^{s-\ell})).
	    \end{equation}
	    
	    For the first sum in (\ref{eq:partial-ell}) we have
	    \begin{align*}
	    	\partial(\sum_{r=0}^\ell&x_n^{k-r}\n^rP^{k-\ell})\ =\\
	    	&=\ x_n^kP^{k-(\ell+1)}\!+\!x_n^{k-1}(P^{k-\ell}+\partial(\n P^{k-\ell}))\!+\ldots+\!x_n^{k-\ell}(\n^{\ell-1} P^{k-\ell}\!+\!\partial(\n^\ell P^{k-\ell}))\\
	    	&+\ x_n^{k-(\ell+1)}\n^\ell P^{k-\ell}.
	    \end{align*}
    
        Using that $P\subset\n$, we have
        $$
        \n^{r-1}P^{k-\ell}+\partial(\n^{r}P^{k-\ell})\ = \ \n^{r-1}P^{k-\ell}+\n^{r-1}P^{k-\ell}+\n^{r}P^{k-(\ell+1)}\ =\ \n^{r}P^{k-(\ell+1)}.
        $$
        
        Hence
        \begin{equation}\label{eq:partialnn}
        	\partial(\sum_{r=0}^\ell x_n^{k-r}\n^rP^{k-\ell})\ =\ \sum_{r=0}^\ell x_n^{k-r}\n^r P^{k-(\ell+1)}+x_n^{k-(\ell+1)}\n^\ell P^{k-\ell}.
        \end{equation}
    
        For the second sum in (\ref{eq:partial-ell}) we have
        \begin{align*}
        	\partial(\sum_{s=\ell+1}^k&x_n^{k-s}\n^\ell P^{k-s}J^{s-\ell})\ =\ x_n^{k-(\ell+1)}\partial(\n^\ell P^{k-(\ell+1)}J)\\
        	&+x_n^{k-(\ell+2)}(\n^\ell P^{k-(\ell+1)}J\!+\!\partial(\n^\ell P^{k-(\ell+2)}J^2))\!+\ldots+\!(\n^\ell PJ^{k-(\ell+1)}+\partial(\n^\ell J^{k-\ell})).
        \end{align*}
    
        We have
        \begin{align*}
            \partial(\n^\ell P^{k-(\ell+1)}J)\ &=\ \n^{\ell-1}P^{k-(\ell+1)}J+\n^\ell P^{k-(\ell+2)}J+\n^\ell P^{k-(\ell+1)}\partial(J)\\
            &=\ \n^\ell P^{k-(\ell+2)}J+\n^\ell P^{k-(\ell+1)}\partial(J),
        \end{align*}
        and using that $J\subset\n P$, $\partial(J)\subset\n$ and equation (\ref{eq:HHZ5}), we obtain
        \begin{align*}
        	\n^\ell&P^{k-(s-1)}J^{s-1-\ell}\!+\!\partial(\n^\ell P^{k-s}J^{s-\ell})\ =\\
        	&=\ \n^\ell P^{k-(s-1)}J^{s-(\ell+1)}+\n^{\ell-1}P^{k-s}J^{s-\ell}+\n^\ell P^{k-(s+1)}J^{s-\ell}+\n^{\ell}P^{k-s}J^{s-(\ell+1)}\partial(J)\\
        	&=\ \n^{\ell-1}P^{k-(s+1)}J^{s-(\ell+1)}[\n P^2+PJ+\n J+\n P\partial(J)]\\
        	&=\ \n^{\ell}P^{k-(s+1)}J^{s-(\ell+1)}[P^2+J+P\partial(J)]\\
        	&=\ \n^{\ell+1}P^{k-s}J^{s-(\ell+1)}.
        \end{align*}
        
        Consequently
        \begin{equation}\label{eq:HHZ7}
        	\begin{aligned}
        		\partial(\sum_{s=\ell+1}^{k}&x_n^{k-s}\n^\ell P^{k-s}J^{s-\ell})\ =\ x_{n}^{k-(\ell+1)}(\n^{\ell}P^{k-(\ell+2)}J+\n^\ell P^{k-(\ell+1)}\partial(J))\\
        		&+(\sum_{s=\ell+2}^kx_{n}^{k-s}\n^{\ell+1}P^{k-s}J^{s-(\ell+1)}).
        	\end{aligned}
        \end{equation}
    
        Using again equation (\ref{eq:HHZ5}), we have
        \begin{equation}\label{eq:HHZ8}
        	\begin{aligned}
        		\n^\ell P^{k-\ell}+\n^{\ell}P^{k-(\ell+2)}J+\n^\ell P^{k-(\ell+1)}\partial(J)\ &=\ \n^{\ell}P^{k-(\ell+2)}[P^2+J+P\partial(J)]\\
        	&=\ \n^{\ell+1}P^{k-(\ell+1)}.
        	\end{aligned}
        \end{equation}
    
        Combining equations (\ref{eq:partial-ell}), (\ref{eq:partialnn}), (\ref{eq:HHZ7}) and (\ref{eq:HHZ8}) we finally conclude that the formula given in equation (\ref{eq:partialI(G)}) for $\partial^{\ell+1}(I^k)$ is correct. In particular, for $\ell+1=k$ we have
        $$
        \partial^k(I^k)\ =\ \sum_{r=0}^kx_n^{k-r}\n^r\ =\ \m^k.
        $$
        
        Therefore, for each $k\le\ell\le 2k$, we clearly have $\partial^{\ell}(I^k)=\m^{2k-\ell}$.
        
        Having established the formula (\ref{eq:partialI(G)}) for all values $1\le\ell\le 2k$, it remains to prove that $\partial^\ell(I^k)$ has linear quotients. This is clear for $k\le\ell\le 2k$ because $\partial^\ell(I^k)=\m^{2k-\ell}$.
        
        Now, let $1\le\ell\le k-1$. Then we can write $\partial^\ell(I^k)=\sum_{t=0}^k x_n^{k-t}I_t$, where $I_t=\n^tP^{k-\ell}$ for $0\le t\le\ell$, and $I_{t}=\n^{\ell}P^{k-t}J^{t-\ell}$ for $\ell+1\le t\le k$. Using that $J\subset\n P$ and $P\subset\n$, we obtain that
        $$
        I_0\supset I_1\supset I_2\supset\cdots\supset I_k.
        $$
        
        We claim that each $I_t$ has linear quotients. To this end, we use a result of Soleyman Jahan and Zheng \cite[Lemma 2.5]{JZ} which says that if $J\subset S$ is a monomial ideal with linear quotients and $Q\subset S$ is a monomial prime ideal such that $\supp(J)\subset\supp(Q)$, then $QJ$ has linear quotients too. Applying this result in our case with $Q=\n$, and using the property (iii), we see that all ideals $I_t$ have indeed linear quotients.
        
        To finish the proof we use a result of Ficarra, Moradi and R\"omer \cite[Lemma 3.2]{FMR} which says that if $I_1,I_2\subset S$ are monomial ideals with linear quotients such that $I_2\subset I_1$, $i\notin\supp(I_2)$ and $\mathcal{G}(x_i I_1)\subset\mathcal{G}(x_iI_1+I_2)$, then the ideal $x_iI_1+I_2$ has again linear quotients.
        
        Applying this result in our situation, since $I_0\subset I_1$, both ideals have linear quotients, $n\notin\supp(I_1)$ and $\mathcal{G}(x_nI_0+I_1)=\mathcal{G}(x_nI_0)\sqcup\mathcal{G}(I_1)$, we conclude that $x_nI_0+I_1$ has linear quotients. For the same reasons, $x_n(x_nI_0+I_1)+I_2$ has linear quotients. Iterating this argument to the other ideals $I_t$, we finally obtain that $\partial^\ell(I^k)$ has indeed linear quotients, concluding the proof.
	\end{proof}

    We expect that Theorem \ref{Thm:edgeideal} actually holds for any monomial ideal (not necessarily squarefree) with $2$-linear resolution.\smallskip

	By \cite[Theorem B]{FM}, if $I\subset S$ is a complementary edge ideal with linear resolution, then $I^k$ has linear resolution, in fact linear quotients, for all $k\ge1$. We expect a stronger result, namely that the powers of complementary edge ideals with linear resolution have differential linear resolution.\medskip
	
	\textbf{Acknowledgment.} A. Ficarra was supported by the Grant JDC2023-051705-I funded by MICIU/AEI/10.13039/501100011033 and by the FSE+ and also by INDAM (Istituto Nazionale Di Alta Matematica). The author is grateful to E. Sgroi for some useful discussion around the topic of this paper.
\end{document}